\documentclass[10pt]{amsart}
\usepackage{graphicx, amsmath, amssymb, amsthm, amsfonts, mathtools, tikz-cd, comment, tikz, caption, subcaption, url, tabularx, enumitem, mathrsfs, amsrefs, appendix, float}

\newtheorem{introthm}{Theorem}
\newtheorem{introcor}[introthm]{Corollary}
\newtheorem{introconj}[introthm]{Conjecture}

\newtheorem{thm}{Theorem}[section]
\newtheorem{lem}[thm]{Lemma}
\newtheorem{prop}[thm]{Proposition}
\newtheorem{cor}[thm]{Corollary}

\theoremstyle{definition}
\newtheorem{defn}[thm]{Definition}

\newtheorem{ex}[thm]{Example}

\newtheorem{rmk}[thm]{Remark}

\newcommand{\dotminus}{\mathbin{\ooalign{\hss\raise1ex\hbox{.}\hss\cr \mathsurround=0pt$-$}}}
\newcommand{\U}{\mathscr{U}}

\title{The Selfless Dichotomy}
\author{Miles Gould}

\address{Department of Mathematics,
University of Louisiana at Lafayette,
Lafayette, USA}

\email{miles.gould1@louisiana.edu}

\date{\today}

\begin{document}

\begin{abstract}
    The purpose of this note is to address the gap in the stable rank one/purely infinite dichotomy of selfless $C^*$-probability spaces. In particular, we show that nonfaithful selfless $C^*$-probability spaces are purely infinite, simple. This completes the dichotomy: Every selfless $C^*$-algebra either has stable rank one or is purely infinite. Notably, this shows that every selfless $C^*$-algebra is pure. To accomplish this, we show that infinite reduced free products of $C^*$-probability spaces with nonfaithful states inducing faithful GNS representations are often purely infinite, simple. Having resolved the dichotomy, we improve existing permanence properties of selfless $C^*$-probability spaces, make progress on a conjecture of Choda and Dykema, and produce several new isomorphisms arising from reduced free products.
\end{abstract}

\maketitle

\section{Introduction}

A unital $C^*$-algebra $A$ has stable rank one if the identity in the matrix algebra $M_n(A)$ is not equivalent to a subprojection, and $A$ admits cancellation of projections. On the other hand, $A$ is called purely infinite if every positive element is paradoxically large, absorbing a copy of itself. We can think of these as the two extreme ends of finite and infinite. Although there are many simple $C^*$-algebras which neither have stable rank one nor are purely infinite, many important classes of simple $C^*$-algebras are expected, or known, to fit into this dichotomy.

Motivated by the breakthroughs in \cite{AMR} and introduced by Robert in 2025 \cite{ROB}, selfless $C^*$-algebras have captured the attention of many researchers \cites{FLO,GAO,HAY,HAY2,OZA,RAU,VIG}. A $C^*$-probability space $(A,\varphi)$ is selfless if there exists a copy of $(A,\varphi)$ inside an ultrapower $(A,\varphi)^\U$ which is freely independent from the diagonal copy. Selflessness provides a testable mechanism through which a $C^*$-algebra has strict comparison, a property in which positive elements are ordered by their state-dimensions.

Selflessness also gives rise to the stable rank one/purely infinite dichotomy. For a selfless $C^*$-probability space $(A,\varphi)$ with tracial $\varphi$, $A$ is monotracial and has stable rank one. If $\varphi$ is nontracial yet faithful, $A$ is purely infinite. The remaining case was that of selfless $(A,\varphi)$ with nonfaithful $\varphi$. We resolve this case.

\begin{introthm}[Theorem \ref{pureinf2}]\label{pureinf}
    Let $(A,\varphi)$ be a selfless $C^*$-probability space with $\varphi$ nonfaithful. Then $A$ is purely infinite, simple.
\end{introthm}

This completes Robert's original dichotomy \cite[3.1]{ROB}.

\begin{introthm}\label{selfdich}
    Let $(A,\varphi)$ be a selfless $C^*$-probability space. Then $A$ is simple.
    \begin{enumerate}[label=(\roman*)]
        \item If $\varphi$ is tracial, then $A$ has stable rank one, strict comparison of positive elements by $\varphi$, and $\varphi$ is the unique (2-quasi)tracial state on $A$.
        \item If $\varphi$ is nontracial, then $A$ is purely infinite.
    \end{enumerate}
\end{introthm}

The full dichotomy justifies the ubiquitous terminology selfless \textit{$C^*$-algebra} rather than selfless \textit{$C^*$-probability space}, as the state involved is either the unique trace or any state makes the $C^*$-algebra selfless, per Ozawa \cite[3]{OZA}, so selflessness is really a property of the algebra. Concretely, we call a $C^*$-algebra $A$ \textit{selfless} if there exists a state $\varphi$ such that $(A,\varphi)$ is selfless. As pointed out to the author by Hannes Thiel, selfless $C^*$-algebras relative to faithful states are pure via this dichotomy \cite{THI}. As an immediate corollary of theorem \ref{selfdich}, we remove the faithful assumption.

\begin{introcor}\label{pure}
    Every selfless $C^*$-algebra is pure.
\end{introcor}

Theorem \ref{selfdich} also provides several checkable sufficient conditions ensuring a reduced free product is purely infinite, simple. Consider the following conjecture of Choda and Dykema \cite{CHODYK}.

\begin{introconj}[Choda-Dykema]
    Let $(A,\varphi)$,$(B,\psi)$ be nontrivial inducing faithful GNS. If $(A,\varphi)*(B,\psi)$ is simple and $\varphi$ or $\psi$ is nontracial, then $(A,\varphi)*(B,\psi)$ is purely infinite, simple.
\end{introconj}

Theorem \ref{selfdich} provides a positive answer when simplicity is strengthened to selflessness. In fact, we only need one factor to be selfless.

\begin{introcor}\label{selfpure}
    Let $(A,\varphi)$ be selfless, $(B,\psi)$ nontrivial inducing faithful GNS. If $\varphi$ or $\psi$ is nontracial, then $(A,\varphi)*(B,\psi)$ is purely infinite, simple. In particular, if $A$ is purely infinite, simple, so is $(A,\varphi)*(B,\psi)$.
\end{introcor}

To prove theorem \ref{pureinf}, we analyze infinite reduced free products of nonfaithful yet GNS-faithful states and obtain the following result.

\begin{introthm}[Theorem \ref{freedich2}]\label{freedich}
    Let $J$ be infinite. For each $j\in J$, let $(A_j,\varphi_j)$ be a $C^*$-probability space with nonfaithful $\varphi_j$ inducing faithful GNS. Define 
    \[I_j:=\mathrm{Ideal}_{A_j}(\ker(\varphi_j)_+)\text{ and }\delta_j:=\|\varphi_j|_{I_j}\|.\] 
    \begin{enumerate}[label=(\roman*)]
        \item Then $\ast_{j\in J}(A_j,\varphi_j)$ contains an essential and purely infinite, simple ideal.
        \item If $\sum_{j\in J}\delta_j>1$, then $\ast_{j\in J}(A_j,\varphi_j)$ is purely infinite, simple.
    \end{enumerate}
\end{introthm}

Remarkably, the conclusion of theorem \ref{freedich} does not depend on some interplay between the factors. It gives a sufficient condition, which is computable at the level of the individual factors, and $\delta_j>0$ in general, so this condition is rather weak. It also makes no demand of unitaries in the kernels/centralizers of the factors, distinguishing it from several prior results in this direction \cites{CHODYK,DYK2,DYK1}.

\subsection*{Acknowledgements}

Part of this work was completed during the Logic and $C^*$-algebras conference (IlijasFest) in Cetraro, Italy as well as the Canadian Operator Symposium 2026 at Carleton University. I would like to thank the organizers for their effort and hospitality. I also thank Ilan Hirshberg, David Jekel, Hannes Thiel, and Stuart White for valuable conversations throughout. I am grateful to Leonel Robert for several crucial ideas which greatly improved this note. As always, I am indebted to Ilijas Farah for suggesting this remarkable model-theoretic problem, and for the opportunities he afforded me during the writing of this note.

\section{Main Results}

First, let us quickly define the objects of our study.

\begin{defn}\label{self}
    A \textit{$C^*$-probability space} is a pair $(A,\varphi)$, where $A$ is a unital $C^*$-algebra and $\varphi$ is a state on $A$. The language of $C^*$-probability spaces consists of sorts $A_n$ (the $n$-balls), functions encoding addition, multiplication, scalar multiplication, and involution, as well as two predicates: $\|\cdot\|$ and $\varphi(\cdot)$. For more exposition, see \cite{FAR},\cite{ROB}.

    We call an embedding $\theta:(B,\psi)\hookrightarrow(A,\varphi)$ \textit{existential} if for every quantifier-free formula $\Phi(\bar{x},\bar{y})$ over tuples $\bar{x},\bar{y}$ in the unit ball, \[\inf_{\bar{y}\in B_1}\Phi(\bar{a},\bar{y})^{(B,\psi)}=\inf_{\bar{y}\in A_1}\Phi(\theta(\bar{a}),\bar{y})^{(A,\varphi)}.\] It is worth mentioning that this is equivalent to the condition that there exists an embedding $\sigma:(A,\varphi)\hookrightarrow(B,\psi)^\U$ for some ultrafilter $\U$ such that $\sigma\theta$ is the diagonal embedding $\iota:B\hookrightarrow B^\U$, though we never use this characterization.
    
    A $C^*$-probability space $(A,\varphi)$ is called \textit{selfless} if $A\neq\mathbb{C}$, $\varphi$ induces faithful GNS, and the first factor embedding $(A,\varphi)\hookrightarrow (A,\varphi)*(A,\varphi)$ is existential.
\end{defn}

Throughout this section, we will prove our main theorems. To do so, we will make use of the Cuntz order. For $a,b\in A_+$, we write $a\precsim b$ if there exists $z_n\in A$ with $\|z_n^*bz_n-a\|\rightarrow 0$. The symmetrization of $\precsim$ is an equivalence relation, and we denote by $[a]$ the class of $a\in A_+$.

\begin{defn}\label{index}
    Let $J$ be infinite. For each $j\in J$, let $(A_j,\varphi_j)$ be a $C^*$-probability space with nonfaithful $\varphi_j$ inducing a faithful GNS representation. Define $(C,\omega):=\ast_{j\in J} (A_j,\varphi_j)$. As $\varphi_j$ is nonfaithful, fix $a_j\in (A_j)_+$ with $\|a_j\|=1$, $\varphi_j(a_j)=0$. For $F\subseteq J$ finite, define $a_F:=\sum_{j\in F}a_j\in C$ after identifying $A_j\subseteq C$. Finally, let $I_j:=\mathrm{Ideal}_{A_j}(\mathrm{ker}(\varphi_j)_+)$ and $\delta_j:=\|\varphi_j|_{I_j}\|$.
\end{defn}

The following lemma is the main driver of theorem \ref{freedich}. We sacrifice the beauty of its statement to parse the steps involved.

\begin{lem}\label{tail}
    Let $J$, $(A_j,\varphi_j)$ be as in the preceding definition. 
    \begin{enumerate}[label=(\roman*)]
        \item For all $i,j\in J$, $a_i\precsim a_j$. That is, $\alpha:=[a_j]$ is independent of $j$.
        \item For all nonzero $b\in C_+$, $m\in\mathbb{N}$, $m\alpha\leq[b]$.
        \item $I:=\mathrm{Ideal}_C(\alpha)$ is the minimal essential ideal of $C$.
        \item $I$ is purely infinite, simple.
    \end{enumerate}
\end{lem}

\begin{proof}
    On the free product Hilbert space
    \[H:=\mathbb{C}\xi\oplus\bigoplus_{k=1}^\infty\bigoplus_{j_1\neq\cdots\neq j_k}\overset{\circ}{H}_{j_1}\otimes\cdots\otimes\overset{\circ}{H}_{j_k},\]
    $a_j$ annihilates the cyclic vector $\xi$ and the subspace of centered words whose first letter does not come from $A_j$ and outputs centered words with first letter from $A_j$. As the free product representation is faithful, $a_ia_j=0$ whenever $i\neq j$.

    Fix $i,j\in J$. If $i=j$, $(i)$ is trivial, so assume $i\neq j$. Since $\varphi_j$ induces faithful GNS, fix $v\in A_j$ such that $t=\omega(v^*a_jv)=\varphi_j(v^*a_jv)>0$. We can express $v^*a_jv=t+c_0$, where $c_0$ is a centered element of $A_j$. Then $a_iv^*a_jva_i=a_i(t+c_0)a_i=ta_i^2$, so $a_i\precsim a_i^2\precsim a_j$.

    For $(ii)$, fix nonzero $b\in C_+$ and $m\in\mathbb{N}$. Again, as the representation of $\omega$ is faithful, fix $v\in C$ such that $t:=\omega(v^*bv)>0$, let $c:=v^*bv$, and note $c\precsim b$. Fix finite $G\subseteq J$ sufficiently large so that there is self-adjoint $d\in \ast_{i\in G}(A_i,\varphi_i)$ with $\|c-d\|<\frac{t}{3}$, so that $\omega(d)\geq\frac{2t}{3}$. We can express $d=\omega(d)+d_0$, where $d_0$ is a limit of a combination of words comprised of centered letters from the factors in $G$. Then for any $F\subseteq J\setminus G$ with $|F|=m$, 
    \[a_Fda_F=a_F(\omega(d)+d_0)a_F=\omega(d)a_F^2\]
    by the same reasoning as in the preceding paragraph. Then 
    \[a_Fca_F\geq a_Fda_F+a_F(c-d)a_F\geq \omega(d)a_F^2-\|c-d\|a_F^2\geq\frac{t}{3}a_F^2,\]
    so $a_F\precsim a_F^2\precsim c\precsim b.$ For any two distinct $i,j\in F$, $a_ia_j=0$, so $[a_F]=m\alpha$. By $(ii)$, if $K$ is a nontrivial ideal of $C$, $b\in K_+\setminus\{0\}$ has $\alpha\leq[b]$, so $I\subseteq K$, so $(iii)$ follows.

    Finally, for $(iv)$, since $a_j$ is full in $I$, for all nonzero $b\in I_+$, $[b]\leq \infty\alpha$. Combining this with $(ii)$, some standard Cuntz semigroup analysis grants $2[b]\leq[b]$, so $I$ is purely infinite, simple.
\end{proof}

This proof illuminates our mechanism for purely infinite simplicity. While nonfaithfulness tends to be a difficulty when proving purely infinite simplicity for finitely many factors, it actually becomes the vehicle when there are infinitely many factors, as we can use tails consisting of positive kernel elements to excise arbitrary positive elements, allowing for robust comparison under $\precsim$.

\begin{thm}\label{freedich2}
    Let $J$ be infinite. For each $j\in J$, let $(A_j,\varphi_j)$ be a $C^*$-probability space with nonfaithful $\varphi_j$ inducing faithful GNS. Define 
    \[I_j:=\mathrm{Ideal}_{A_j}(\ker(\varphi_j)_+)\text{ and }\delta_j:=\|\varphi_j|_{I_j}\|.\] 
    \begin{enumerate}[label=(\roman*)]
        \item Then $\ast_{j\in J}(A_j,\varphi_j)$ contains an essential and purely infinite, simple ideal.
        \item If $\sum_{j\in J}\delta_j>1$, then $\ast_{j\in J}(A_j,\varphi_j)$ is purely infinite, simple.
    \end{enumerate}
\end{thm}

\begin{proof}
    Let $\varphi_j$ be nonfaithful. It suffices to show that $I$ is full in $C=\ast_{j\in J}(A_j,\varphi_j)$. First, note by \ref{tail}$(iii)$ that $I$ is the unique minimal essential ideal of $C$, so is independent of the choice of $a_j$ in definition \ref{index}. Therefore, $\ker(\varphi_j)_+\subseteq I$ and in particular, $I_j\subseteq I$ for all $j\in J$.
    
    Choose finite $F\subseteq J$ such that $\sum_{j\in F}\delta_j>1$. Then, for each $j\in F$, there exists $e_j\in (I_j)_+$ such that $\|e_j\|=1$ and $\sum_{j\in F}\varphi_j(e_j)>1$. Suppose, towards contradiction, that $c:=\sum_{j\in F}e_j\not\geq \epsilon1_C$ for any $\epsilon>0$. Then there are unit vectors $\eta_n\in H$ such that $\langle c\eta_n,\eta_n\rangle\rightarrow 0$, so $\|e_j\eta_n\|^2=\langle e_j^2\eta_n,\eta_n\rangle\leq \langle e_j\eta_n,\eta_n\rangle\rightarrow 0$ for all $j\in F$. Let $q_j:=1_C-e_j$. 
    
    Consider the subspace
    \[H(j):=\mathbb{C}\xi\oplus\bigoplus_{k=1}^\infty\bigoplus_{j\neq j_1\neq\cdots\neq j_k}\overset{\circ}{H}_{j_1}\otimes\cdots\otimes\overset{\circ}{H}_{j_k}.\]
    There is a canonical isomorphism $H\cong H_j\otimes H(j)$ by $\xi_j\otimes \eta\mapsto \eta$ and $\zeta\otimes \eta\mapsto\zeta\otimes \eta$ for $\zeta\in \overset{\circ}{H_j}$, $\eta\in H(j)$, where $\xi_j$ is the cyclic vector associated to the $j$-th factor. Let $E_j$ be the projection of $H$ onto $H(j)$, given by $E_j:=P_{\mathbb{C}\xi_j}\otimes 1$. Hence $\|E_jq_j\|=\|P_{\mathbb{C}\xi_j}q_j\|=\|q_j\xi_j\|$, so
    \[\|E_j\eta_n\|\leq \|E_jq_j\eta_n\|+\|E_je_j\eta_n\|\leq\|q_j\xi_j\|+\|e_j\eta_n\|,\]
    which grants \[\limsup_{n\rightarrow\infty}\|E_j\eta_n\|^2\leq\|q_j\xi_j\|^2=\varphi_j((1-e_j)^2)\leq 1-\varphi_j(e_j).\]
    Fix an arbitrary unit vector $\eta\in H$ and express $\eta=\eta^0+\sum_{j\in F}\eta^j+\eta^c$, where $\eta^0\in\mathbb{C}\xi$, $\eta_j$ begins with first letter from the $j$-th factor, and no component of $\eta^c$ has first letter in the $F$-th factors. As $E_j$ only annihilates $\eta^j$, then $\|E_j\eta\|^2=1-\|\eta^j\|^2$, so
    \[\sum_{j\in F}\|E_j\eta\|^2=|F|-\sum_{j\in F}\|\eta^j\|^2\geq |F|-1.\]
    Then for $\eta=\eta_n$,
    \[|F|-1\leq \limsup_{n\rightarrow\infty}\sum_{j\in F}\|E_j\eta_n\|^2\leq\sum_{j\in F}\limsup_{n\rightarrow\infty}\|E_j\eta_n\|^2\leq\sum_{j\in F}(1-\varphi_j(e_j)),\]
    so $\sum_{j\in F}\varphi_j(e_j)\leq 1$, a contradiction. Hence $c\geq \epsilon 1_C$ for some $\epsilon>0$, so $1_C\precsim c$. Therefore $1_C\in I$, so by \ref{tail}$(iv)$, $C=I$ is purely infinite, simple.

\end{proof}

Theorem \ref{pureinf} follows by model theory.

\begin{thm}\label{pureinf2}
    Let $(A,\varphi)$ be a selfless $C^*$-probability space with $\varphi$ nonfaithful. Then $A$ is purely infinite, simple.
\end{thm}

\begin{proof}
    Assume $(A,\varphi)$ is selfless with $\varphi$ nonfaithful. As $\delta_j>0$ is constant, $\sum_{j=1}^\infty\delta_j=\infty>1$, so by theorem \ref{freedich}, $(C,\omega):=\ast_{j=1}^\infty(A,\varphi)$ is purely infinite, simple. As noted in \cite[3.13.7]{FAR}, purely infinite, simple $C^*$-algebras are axiomatized by
    \[\sup_{x\in (C_+)_1}\sup_{y\in (C_+)_1}\min\{\|x\|\dotminus\frac{1}{2},\inf_{z\in C_1}\|2zxz^*-y\|\}=0,\]
    where $s\dotminus t:=\max\{s-t,0\}$. By \cite[2.6(vi)]{ROB}, the first factor embedding $\theta:(A,\varphi)\hookrightarrow(C,\omega)$ is existential. Fix $a,b\in (A_+)_1$. Then
    \[\inf_{z\in A_1}\min\{\|a\|\dotminus\frac{1}{2},\|2zaz^*-b\|\}=\inf_{z\in C_1}\min\{\|\theta(a)\|\dotminus\frac{1}{2},\|2z\theta(a)z^*-\theta(b)\|\}=0.\]
    As $a,b\in (A_+)_1$ were arbitrary, $A$ is purely infinite, simple.
\end{proof}

Note that at the beginning of the proof of theorem \ref{pureinf}, we use the corollary of theorem \ref{freedich} that for nonfaithful $\varphi$, $\ast_{j=1}^\infty(A,\varphi)$ is purely infinite, simple. However, we refrain from stating this corollary on its own, as theorem \ref{selfdich} provides an immediate self-strengthening.

\begin{cor}
    Let $(A,\varphi)$ be a $C^*$-probability space with $\varphi$ nontracial, inducing faithful GNS. Then $\ast_{j=1}^\infty(A,\varphi)$ is purely infinite, simple.
\end{cor}

\begin{proof}
    Apply \cite[2.2]{ROB} to see that $\ast_{j=1}^\infty(A,\varphi)$ is selfless and theorem \ref{selfdich} to see it is purely infinite.
\end{proof}

\section{Permanence Properties}

Keeping in the theme of cleaning up existing results, we strengthen a few existing permanence properties of selfless $C^*$-probability spaces by removing assumptions. In the next proposition, we remove the assumption that $B$ be separable in \cite[4.2]{ROB}.

\begin{prop}\label{selfabs}
    Let $(A,\varphi)$ be a selfless $C^*$-probability space, $(B,\psi)$ inducing faithful GNS. Then $(A,\varphi)*(B,\psi)$ is selfless.
\end{prop}

\begin{proof}
    If $B$ is separable, then \cite[4.2]{ROB} grants us the conclusion.
    Otherwise, express $(B,\psi)$ as the direct limit of its separable elementary submodels $(B_\lambda,\psi_\lambda)$, where $\psi_\lambda:=\psi|_{B_\lambda}$. 
    
    We claim that $\psi_\lambda$ is GNS-faithful. Let $\theta:(B_\lambda,\psi_\lambda)\hookrightarrow(B,\psi)$ be the elementary embedding. Fix $a\in B_\lambda$ and suppose $\pi_{\psi_\lambda}(a)=0$. Then
    \begin{align*}
        0=\|\pi_{\psi_\lambda}(a)\|\geq\sup_{y\in (B_\lambda)_1}\|ay\|_\psi=\sup_{y\in B_1}\|\theta(a)y\|_\psi.
    \end{align*}
    Thus, for all $y\in B,$ $\|\theta(a)y\|_\psi=0$, so $\|\pi_\psi(\theta(a))\|=0$. As $\pi_\psi$ is faithful, $\theta(a)=0$, so $a=0$. Since $a\in B_\lambda$ was arbitrary, $\pi_{\psi_\lambda}$ is faithful. Hence, $(A,\varphi)*(B_\lambda,\psi_\lambda)$ is selfless. Then, as $(A,\varphi)*(B,\psi)$ is the direct limit of $(A,\varphi)*(B_\lambda,\psi_\lambda)$, $(A,\varphi)*(B,\psi)$ is selfless by \cite[4.1]{ROB}.
\end{proof}

We waited to prove corollary \ref{selfpure} until this point so that we could remove the separability assumption on $B$. Regardless, it follows immediately by proposition \ref{selfabs} and theorem \ref{selfdich}. We can also remove the faithful assumption in \cite[4.3]{ROB}.

\begin{prop}
    Let $(A,\varphi)$ be a selfless $C^*$-probability space. Then for every $n\in\mathbb{N}$, $(M_n(A),\varphi\otimes\mathrm{tr}_n)$ is selfless. In particular, for every UHF $(B,\tau)$, where $\tau$ is the unique trace, $(A\otimes B,\varphi\otimes \tau)$ is selfless.
\end{prop}

\begin{proof}
    If $\varphi$ is faithful, the conclusion follows by \cite[4.3]{ROB}. If $\varphi$ is nonfaithful, by theorem \ref{pureinf}, $A$ is purely infinite, simple, so $M_n(A)$ is as well. By \cite[3]{OZA}, $(M_n(A),\varphi\otimes\mathrm{tr}_n)$ is selfless. The last claim follows by a direct limit.
\end{proof}

\begin{prop}
    Let $(A,\varphi)$ be a selfless $C^*$-probability space and $p$ a nonzero projection in $A$ such that $\varphi(p)>0$. Then $(pAp,\frac{1}{\varphi(p)}\varphi)$ is selfless. If $\varphi(p)=0$, then for any state $\psi$, $(pAp,\psi)$ is selfless.
\end{prop}

\begin{proof}
    If $\varphi$ is tracial, then \cite[4.4]{ROB} yields the conclusion. If $\varphi$ is nontracial, then $A$ is purely infinite, simple by theorem \ref{selfdich}. In particular, $pAp$ is purely infinite, simple as well. Therefore, by \cite[3]{OZA}, $(pAp,\frac{1}{\varphi(p)}\varphi)$ is selfless.
\end{proof}

Unital $C^*$-algebras $A$ and $B$ are Morita equivalent if and only if there exists $n\in\mathbb{N}$ and a full projection $p\in M_n(A)$ such that $B\cong pM_n(A)p$. It is thus natural to say $(A,\varphi)$ and $(B,\psi)$ are Morita equivalent if $B\cong pM_n(A)p$ and $\psi=\frac{1}{\varphi\otimes\mathrm{tr}_n(p)}\varphi\otimes\mathrm{tr}_n$.

\begin{cor}
    Let $(A,\varphi)$ be a selfless $C^*$-probability space. If $(B,\psi)$ is Morita equivalant to $(A,\varphi)$, then $(B,\psi)$ is selfless.
\end{cor}

\section{Isomorphisms of Reduced Free Products}

One powerful application of theorem \ref{freedich} and corollary \ref{selfpure} is in supplying Kirchberg algebras satisfying the UCT.

\begin{thm}\label{UCT}
    For $j\in\mathbb{N}$, let $A_j\neq\mathbb{C}$ be a separable, nuclear $C^*$-algebra satisfying the UCT, $\varphi_j$ inducing faithful GNS, and $\varphi_j$ pure for $j\geq 2$. Then $C:=\ast_{j=1}^\infty(A_j,\varphi_j)$ is a Kirchberg $C^*$-algebra satisfying the UCT. If $\omega$ is a pure state on $C$, then $(C,\omega)\cong\ast_{j=1}^\infty(A_j,\varphi_j)$ if and only if $\varphi_1$ is also pure.
\end{thm}

\begin{proof}
    First, we note that by Ozawa \cite[1.1]{OZANUC}, 
    \[\ast_{j=1}^n(A_j,\varphi_j)=(A_1,\varphi_1)*\ast_{j=2}^n(A_j,\varphi_j)\] is nuclear, since $\ast_{j=2}^n\varphi_j$ is pure by \cite[1.6.5]{VOI}. Moreover, by \cite[1.1]{HAS} and \cite[2.7]{THO}, $\ast_{j=1}^n(A_j,\varphi_j)$ satisfies the UCT. Then, separable nuclear $C^*$-algebras satisfying the UCT are closed under countable direct limits \cite[22.3.4]{BLA}, so it suffices to show that $C$ is purely infinite, simple.
    
    We claim that, for $j\geq 2$, $(A_j,\varphi_j)$ has $\delta_j=1$. In particular, we show that $(A,\varphi):=(A_j,\varphi_j)$ has $\|x\|\leq\sup_{y\in A_1}\|xy\|_\varphi$. Let $\pi_\varphi:A \rightarrow\mathscr{B}(H_\varphi)$ be the GNS-representation. By \cite[III.2.16]{TAKE}, since $\varphi$ is pure, we have $\|x+L_\varphi\|=\|x\|_\varphi$. Hence, for each $\eta\in H_\varphi$ with $\|\eta\| < 1$, we can find $y\in A_1$ s.t. $\eta=y+L_\varphi$. Since these vectors are dense in $(H_\varphi)_1:=\{\eta\in H_\varphi:\|\eta\|\leq1\}$, we obtain \[\|x\|=\|\pi_\varphi(x)\|=\sup_{\eta\in (H_\varphi)_1}\|\pi_\varphi(x)\eta\|_\varphi\leq\sup_{y\in A_1}\|xy\|_\varphi\]
    by the faithfulness of $\pi_{\varphi}$. Thus $\delta_j\geq\sup_{v\in (A_j)_1}\varphi_j(v^*a_jv)=1$.
    
    Since pure states are nonfaithful, by theorem \ref{freedich}, $\ast_{j=2}^\infty(A_j,\varphi_j)$ is purely infinite, simple, so by corollary \ref{selfpure}, C is a Kirchberg algebra satisfying the UCT. The last claim follows by \cite[1.6.5]{VOI} and the homogeneity of the pure state space \cite{KIS}.
\end{proof}

\begin{cor}\label{UCT2}
    Let $A,B\neq\mathbb{C}$ be separable, nuclear, satisfying the UCT, $\varphi,\psi$ inducing faithful GNS. If $(A,\varphi)$ is selfless and $\varphi$ or $\psi$ is pure, then $C:=(A,\varphi)*(B,\psi)$ is a Kirchberg $C^*$-algebra satisfying the UCT.
    If $\omega$ is a pure state on $C$, then $(C,\omega)\cong(A,\varphi)*(B,\psi)$ if and only if $\varphi$ and $\psi$ are pure.
\end{cor}

\begin{proof}
    Apply corollary \ref{selfpure} and the reasoning in the proof of theorem \ref{UCT}.
\end{proof}

The conclusions of \ref{UCT} and \ref{UCT2} grant many new isomorphisms. Indeed, by the Kirchberg-Phillips classification theorem \cite[8.2.1]{ROR}, one need only compute unital K-theory to determine isomorphism of these reduced free products. To compute the K-theory, use \cite[1.1]{HAS} and \cite[6.4]{THO}. Some examples include $\mathcal{O}_\infty\cong(\mathcal{Z},\tau)*(\mathcal{Z},\varphi)$ and
$(\mathcal{O}_\infty,\omega)\cong\ast_{j=1}^\infty(\mathcal{Z},\varphi)$ for any pure states $\varphi,\omega$ and $\tau$ the unique trace. It also partially resolves multiple questions in section 4 of \cite{HIR}. Indeed, let $A$ be separable, nuclear, satisfying the UCT, $\varphi$ a pure state on $A$, $\psi$ a pure state on $\mathcal{O}_\infty$, and $\lambda$ the Lebesgue measure on $[0,1]$. Then we have
\[(\mathcal{O}_\infty,\psi)\ast\ast_{j=1}^\infty(A,\varphi)\cong\ast_{j=1}^\infty(A,\varphi)\cong (C([0,1]),\lambda)\ast\ast_{j=1}^\infty(A,\varphi)\]
where the first isomorphism is state preserving, but the latter necessarily is not. In fact, this is stronger as we need not assume $A$ is purely infinite, simple.

\end{document}